\documentclass{amsart}
\usepackage[utf8]{inputenc}
\usepackage[english]{babel}
\usepackage{pgfplots}
\pgfplotsset{compat=1.15}
\usepackage{mathrsfs}
\usepackage{pstricks-add}
\usepackage{float}
\usepackage{wrapfig}
\usepackage{hyperref}
\usepackage[width=0.3333\textwidth]{caption}
\usetikzlibrary{arrows}

\usepackage{blindtext}
\usepackage{amssymb}

\usepackage{amsmath}
\usepackage{mathtools}
\DeclarePairedDelimiter\norm{\lVert}{\rVert}
\DeclarePairedDelimiter\abs{|}{|}
\usepackage{amsthm}
\usepackage{geometry}
\newtheorem{theorem}{Theorem}[section]
\newtheorem{corollary}[theorem]{Corollary}
\newtheorem{lemma}[theorem]{Lemma}
\newtheorem{proposition}[theorem]{Proposition}
\newtheorem{observation}[theorem]{Observation}
\usepackage{setspace}
\usepackage{xcolor}
\theoremstyle{remark}
\newtheorem*{remark}{Remark}
\newtheorem*{remarks}{Remarks}
\usepackage{graphicx}
\usepackage[shortlabels]{enumitem}
\usepackage[normalem]{ulem}
\usepackage{cancel}
\newcommand{\comment}[1]{}

\theoremstyle{definition}
\newtheorem{definition}{Definition}[section]
\newtheorem{notation}{Notation}

\begin{document}

\title[Tiling of rectangles with squares via Diophantine approximation]
{Tiling of rectangles with squares and related problems via Diophantine approximation}

\author[T.~Keleti]{Tam\'as Keleti}
\address{Institute of Mathematics, E\"otv\"os Lor\'and University, P\'azm\'any P\'eter S\'et\'any 1/c, H-1117 Budapest, Hungary}
\email{tamas.keleti@gmail.com}
\thanks{This research started at a ``Conjecture and Proof''
course of BSM (Budapest Semester in Mathematics)
and continued at a research course of BSM. 
The first author is grateful to the Alfr\'ed R\'enyi Institute, where he was a visiting researcher during the completion of this project.
He was also supported by the Hungarian National Research, Development and Innovation Office - NKFIH, 124749 and 129335.}

\author[S.~Lacina]{Stephen Lacina}
\address{Department of Mathematics, University of Oregon, Fenton Hall, Eugene, Oregon}
\email{slacina@uoregon.edu}

\author[C.~Liu]{Changshuo Liu}
\address{Department of Mathematics, Princeton University, Fine Hall, New Jersey, United States.}
\email{cl20@alumni.princeton.edu}

\author[M.~Liu]{Mengzhen Liu}
\address{Department of Pure Mathematics and Mathematical Statistics, University of Cambridge. CB3 0WB, Centre for Mathematical Sciences, Wilberforce Road, Cambridge, UK.}
\email{ml868@cam.ac.uk}

\author[J.~R.~Tuir\'an Rangel]{Jos\'e Ram\'on Tuir\'an Rangel}
\address{Facultad de Ciencias, Universidad Nacional Aut\'onoma de M\'exico, Ciudad Universitaria, Mexico City, Mexico.}
\email{joseramont@ciencias.unam.mx}




\begin{abstract}
This article shines new light on the classical problem of tiling rectangles with squares efficiently with a novel method. With a twist on the traditional approach of resistor networks, we provide new and improved results 
on the matter using the theory of Diophantine Approximation, hence overcoming long-established difficulties, such as generalizations to higher-dimensional analogues. The universality of the method is demonstrated through its applications to different tiling problems. These include tiling rectangles with other rectangles, with their respective higher-dimensional counterparts, as well as tiling equilateral triangles, parallelograms, and trapezoids with equilateral triangles.
\end{abstract}
\maketitle

\section{Introduction}


Tiling rectangles by squares has been a popular topic for more than 100 years. The first fundamental result is 
the following.

\begin{theorem}[M. Dehn \cite{Dehn}, 1903] \label{maxdehntheorem}
If a rectangle is tiled by squares then
the 
{aspect ratio} of the rectangle is rational.
\end{theorem}

A variety of proofs of Theorem \ref{maxdehntheorem} 
have been discovered since then, with a particular standout being from a paper by R. L. Brooks et al. \cite{BSST} in 1940. This paper presents a one-to-one correspondence between a squared rectangle and a resistor network, with any solution to Kirchoff's equations being exactly the sizes of the square tiles. Using the physical theory of resistor networks, the problem is easily transformed into a graph theoretic one. In fact, Dehn proved 
a stronger version of Theorem~\ref{maxdehntheorem} (later referred to as  ``Strong Dehn's Theorem"
) \cite{Dehn}: If a rectangle is tiled by squares then the side lengths of the squares are all rational
multiples of the side lengths of the rectangle.

A natural next step of development is to investigate the possible rational numbers that could represent the ratio of the sides of a particular tiling; for example, is there anything one could conclude about the ratio given the number of tiles used? Specifically, two versions of this problem are well-studied historically:
\begin{enumerate}
    \item $[$Coprime Rectangle Tiling$]$ How many square tiles of integer side lengths are needed to tile a rectangle of size $p\times q$, where $p,q\in \mathbb N^+$ and $(p,q)=1$?
    \item $[$Prime Tiling, also known as \textit{Mrs Perkins's quilt}$]$ How many square tiles of jointly coprime (integer) side lengths are needed to tile a rectangle of size $p\times q$, where $p,q\in \mathbb N^+$?
\end{enumerate}

It should be more or less clear that Prime Tiling is a more general problem than Coprime Rectangle Tiling, as any integer tiling of a coprime rectangle $p\times q$ cannot have common factors among the side lengths of its tiles: otherwise this would also be a common factor between $p$ and $q$. Therefore, any result on Prime Tiling automatically applies to Coprime Rectangle Tiling. 
For each of these two 
questions, many lower and upper bounds on the number of tiles needed were established in the past endeavors.

Regarding the Coprime Rectangle Tiling problem, Kenyon \cite{Kenyon} (1996) continues to utilize the theory of resistor networks to give an upper bound on the number of tiles needed. It is proved that for any coprime $p,q$ with $p>q$, 
a $p\times q$ rectangle can be tiled using less than $p/q+C\log(p)$ squares with integer side lengths.

As for the Prime Tiling problem (hence also Coprime Rectangle Tiling), Conway \cite{Conway} (1964) established 
a lower bound on the number of tiles for a Prime Tiling using the resistor approach yet again: a 
prime tiling of a $p\times q$ rectangle requires at least $\log_2(\max\{p,q\})$ tiles. 
Conway's result has been recently improved by D.~Fomin \cite{Fomin}, who proved the lower bound $\log_\tau(\max\{p,q\})$, where $\tau\approx 1.8637\ldots$.




However, as noted 
at the end of both \cite{BSST} and \cite{Kenyon}, the very natural higher dimensional, or even $3$-dimensional, analogues of the above two 
problems of tiling hypercuboids with hypercubes are very difficult to attack. This is largely 
due to the fact that ``there is no satisfactory analogue [of resistor networks] in three dimensions,'' as stated in \cite{BSST}. 
One of the best results in higher dimensions so far is 
found in \cite{Walters} by M. Walters in 2009, who proves an upper bound of $(\log p)^{O(d)}$ and a lower bound of $\Omega((\log p)^{d/2})$ for the $d$-dimensional analogue of the Coprime Rectangle Tiling problem where $p$ denotes the longest side length. While this upper bound could be worked into an 
upper bound in the higher dimensional Prime Tiling problem, the lower bound simply does not extend to the higher dimensional Prime Tiling problem as we shall see later in this introduction in a counterexample of a prime tiling of hypercubes whose 
number of tiles is linear in $\log(p)$. 
One will 
also notice that the lower bound in \cite{Walters} is obtained by taking a $2$-dimensional cross section of the hypercube tiles, circumventing but not resolving the issue noticed above in \cite{BSST} and \cite{Kenyon}.



Our goal is to give lower bounds to the Prime Tiling problem, i.e. \textit{Mrs Perkins's quilt}, both in $2$-dimensions and higher dimensions, as well as lower bounds for 
natural related questions about some more general tilings. To achieve this, we first rephrase the 
problem at hand, 
switching to a completely new perspective, and introduce novel techniques to the problem to avoid the known difficulties.

We formally state the Prime Tiling problem, i.e. \textit{Mrs Perkins's quilt} here again:\\
$(1)$ \emph{What is the minimal number of square tiles with integer side lengths required to tile a $p\times q$ rectangle, given that the lengths of the square tiles are jointly coprime?}

While this is a perfectly valid and natural problem, we study an equivalent one that is phrased quite differently:\\
$(2)$ \emph{Suppose a rectangle is tiled by $n$ squares. How large is the minimal-sized scaling of the rectangle 
such that 
all of its square tiles 
have integer side lengths?}

These are equivalent questions in the sense that, a lower bound for the minimal number of square tiles required is equivalent to an upper bound on the size of the scaling. For example, the above mentioned $\log_2(\max\{p,q\})$ estimate of Conway in \cite{Conway} is equivalent to the statement that
if a rectangle is tiled by $n$ squares then there is a scaling of it with longest side length at most $2^n$ such that all square tiles have integer side lengths. In fact, (2) can even be considered as a quantitative strong Dehn's Theorem, while estimates for (1) only imply estimates for (2) by using strong Dehn's theorem.

The difference in 
the perspectives of these equivalent problems 
allows us to think outside of the box and apply new
techniques that otherwise would never be considered. In our paper, we develop such a technique through a brand new approach to this classical problem, and apply it in various settings, demonstrating its universality in dealing with such tiling problems. In the classical case of \textit{Mrs Perkins's quilt} we prove
the following:

\begin{theorem}\label{intro_tilebysquare}
For any 
{tiling of a rectangle by n squares}, one can find a scaling of it with longest side length at most $4^{n}$ such that all squares have integer side lengths. 
\end{theorem}

This result translates to a lower bound of $\frac{1}{2}\log_2(\max\{p,q\})$ in the Prime Tiling problem. Although the new machinery 
gives weaker results in this case compared to Conway and Fomin, we shall see that it is able to break many barriers which 
the classical method 
has thus far failed to overcome.

\comment{{
As a corollary, we obtain a partial answer to a question on MathOverflow \cite{MOPost}. Let $f(p,q)$ denote the least number of square tiles of integer side-lengths needed in a tiling. One clearly has $f(kp,kq)\leq f(p,q)$ when $k\in \mathbb N$; however it is not clear if the reverse inequality holds. While some convincing yet unverified candidates of counterexamples were presented in \cite{MOPost} suggesting a possible negative to the claim, we provide a weak positive partial result:}
\begin{corollary}
{
Let $p\geq q$, $p,q\in \mathbb N$ and $(p,q)=1$. Suppose $f(kp,kq)> n$ for all $k\in \mathbb N$ such that $kp\leq 4^n$, then $f(kp,kq)> n$ for all $k\in \mathbb N$.} 
\end{corollary}

{
In particular, if $f(p,q)=f(kp,kq)$ for all $k\leq 4^{f(p,q)-1}/p$, then $f(p,q)=f(kp,kq)$ for all $k\in\mathbb N$. This result hence restricts the search for counterexamples to a finite problem.}}

In order to prove Theorem~\ref{intro_tilebysquare},
we demonstrate the following result, which seems to be interesting in itself:

\begin{theorem}\label{intro_keycorollary}
Suppose a rectangle 
{is tiled by squares and both $x$-coordinates 
or both $y$-coordinates of the vertices of the rectangle are integers.}
If each coordinate of every vertex of each square tile 
has distance less than $\frac 14$ to
an integer, then all square tiles have integer side lengths.
\end{theorem}

In order to use Theorem~\ref{intro_keycorollary} to obtain Theorem~\ref{intro_tilebysquare}, we 
will use the Simultaneous Version of Dirichlet's Approximation Theorem (Theorem~\ref{approximation})
to re-scale the tiled rectangle and ensure that
all coordinates of all tiling squares 
have distance less than $\frac 14$ to 
an integer.
To optimize the scaling factor as much as we can, we provide a sharp estimate (Lemma~\ref{l:coordvaluesforrectangles}) for the number of distinct coordinates in a rectangle tiling.

Although we prove Theorem~\ref{intro_keycorollary} from scratch, the proof was partly inspired 
by the resistor network technique mentioned above (see more explanation before Lemma~\ref{l:insteadofnetworks}).

In fact, the previously mentioned result of Fomin \cite{Fomin} is more general:
he proved that if a rectangle of size $p\times q$ is tiled by smaller rectangles of side length ratio $p_i/q_i$ with $(p_i,q_i)=1$ $(i=1,\ldots,n)$ such that
the side lengths of the rectangle tiles are jointly coprime then 
\begin{equation*}
    \max\{p,q\} \le \prod_{i=1}^n \tau^{\max\{p_i,q_i\}}, 
    \textrm{ where } \tau\approx 1.8637.
\end{equation*}
With our approach, we prove the following result 
which is weaker than the above result of Fomin when $\max\{p_i,q_i\}\le 6$ for most $i$, but much stronger otherwise
since $8x<\tau^x$ for $x\ge 7$.

\begin{theorem}\label{intro_scaleqrect}
Given a 
{tiling of a rectangle by n rectangles} with side length ratios 
$p_i/q_i$ for $p_i,q_i\in \mathbb{N}^{+}$ $(i=1,\ldots,n)$,
one can find a scaling of it with longest side length at most $$
{8}^{n}\prod_{i=1}^n\max\{p_i,q_i\}
$$ 
such that all tiling rectangles have integer side lengths.
\end{theorem}

\comment{
Of course, a rectangle with side length ratio $p/q$ can be trivially tiled into $pq$ squares, so an upper bound on the length of the longest side of the big rectangle also follows directly from Theorem~\ref{intro_tilebysquare}. 
But that would give an upper bound $4^{p_1q_1+\ldots+p_n q_n}$, which is in most cases much worse than the estimate of Theorem~\ref{intro_scaleqrect}.
}


As for the higher dimensional analogue of the prime tiling problem, the lack of a higher-dimensional analogue of the resistor network technique made it hard for R.~L.~Brooks et al.~\cite{BSST} and Richard Kenyon \cite{Kenyon} to generalize their results as mentioned above. On the contrary, the Diophantine approximation technique offers a simple solution: as long as we scale the tiling such that a $2$-dimensional projection of the tiling has only 
squares with integer side lengths, all hypercubes have integer side lengths. 
This way we will be able to prove the following:

\begin{theorem}\label{h:intro_hypercubesbound}
For any 
{tiling of a hypercuboid by $n$ hypercubes}, one can find a scaling of it with 
shortest side length at most $4^{\frac{2(n-1)}{d}+1}$
and longest side length at most $n\cdot 4^{\frac{2(n-1)}{d}+1}$ such that all tiling hypercubes have integer side lengths.
\end{theorem}

This roughly translates to a lower bound in the classical higher dimensional Prime Tiling problem of $\frac{d}{4}\log_2(p)$, which is not only a result in uncharted territory, 
it is also linear in $\log(p)$, which is tight up to a factor as we will see in an example later.

Finally, we study tilings with equilateral triangles, focusing on parallelograms and equilateral triangles tiled by equilateral triangles. Three analogous problems arise in these types of tilings:
\begin{enumerate}
    \item $[$Coprime Parallelogram Tiling$]$ How many equilateral triangle tiles of integer side lengths are needed to tile a parallelogram with side lengths $p$ and $q$, where $p,q\in \mathbb N^+$ and $(p,q)=1$?
    \item $[$Parallelogram Prime Tiling$]$ How many equilateral triangle tiles of jointly coprime (integer) side lengths are needed to tile a parallelogram with side lengths $p$ and $q$, where $p,q\in \mathbb N^+$?
    \item $[$Triangle Prime Tiling$]$ How many equilateral triangle tiles of jointly coprime (integer) side lengths are needed to tile an equilateral triangle with side length $p \in \mathbb{N}^{+}$?
\end{enumerate}
By making use of the theory presented in \cite{Tutte} (specifically the equations (24) and (25) of page 471), one 
can apply the same techniques offered by \cite{Kenyon} and \cite{Conway} (with the Leaky Electrical Network defined in \cite{BSST75} and its respective matrix) obtaining a lower bound 
of $n \geq \log_2 (p)$ for Triangle Prime Tiling and 
$n \geq \log_2 (\max\{p,q\})$ 
for the parallelogram problems 
where $n$ is 
the number of tiles used in the triangulation. Our estimates, which we present in the next theorem, are slightly worse than these lower bounds, however, we offer them here as it is a natural generalization of our technique in rectangular problems.

\begin{theorem}\label{t:intro_strongt}
For any 
{tiling of an equilateral triangle by $n$ equilateral triangles}, one can find a scaling of it with side length at most $4^{\frac{2n-2}{3}}$ such that all triangle tiles have integer side lengths.
\end{theorem}

As before, this theorem gives a lower bound on the Prime Tiling problem of the form $n \geq \frac{3}{4} \log_2 (p)$, where $p \in \mathbb{N}^{+}$ is the side length of the triangle.\\
\\
How sharp are the above results?
Theorem~\ref{intro_keycorollary} is sharp in the sense that the 
constant $\frac14$ cannot be improved, we cannot even allow equality:
one can decompose the square $[-1,1]\times [-1,1]$ into squares of side-lengths $3/4, 1/2$ and $1/4$ 
such that the coordinates of the vertices of the square tiles are $0, \pm 1/4, \pm 3/4$ and $\pm 1$
(see Figure \ref{fig:counterexample3}).

\begin{figure}[ht]
\begin{minipage}{0.33\textwidth}
\centering
\begin{tikzpicture}[line cap=round,line join=round,>=triangle 45,x=1.5cm,y=1.5cm]
\clip(-1.4,-1.4) rectangle (1.2,1.2);
\draw [line width=1pt] (-0.25,1)-- (-0.25,-1);
\draw [line width=1pt] (0.25,1)-- (0.25,-1);
\draw [line width=1pt] (-1,1)-- (1,1);
\draw [line width=1pt] (1,1)-- (1,-1);
\draw [line width=1pt] (1,-1)-- (-1,-1);
\draw [line width=1pt] (-1,-1)-- (-1,1);
\draw [line width=1pt] (-1,0.25)-- (1,0.25);
\draw [line width=1pt] (-1,-0.25)-- (1,-0.25);
\draw [line width=1pt] (-0.75,0.25)-- (-0.75,-0.25);
\draw [line width=1pt] (0.75,0.25)-- (0.75,-0.25);
\draw [line width=1pt] (-0.25,0.75)-- (0.25,0.75);
\draw [line width=1pt] (-0.25,-0.75)-- (0.25,-0.75);
\draw [line width=1pt] (-1,0)-- (-0.75,0);
\draw [line width=1pt] (0,1)-- (0,0.75);
\draw [line width=1pt] (0.75,0)-- (1,0);
\draw [line width=1pt] (0,-0.75)-- (0,-1);
\draw [line width=0.5pt] (-1.03,-0.75)-- (-0.97,-0.75);
\draw [line width=0.5pt] (-1.03,0.75)-- (-0.97,0.75);
\draw [line width=0.5pt] (-1.03,-0.5)-- (-0.97,-0.5);
\draw [line width=0.5pt] (-1.03,0.5)-- (-0.97,0.5);
\draw [line width=0.5pt] (-1.03,-0.25)-- (-0.97,-0.25);
\draw [line width=0.5pt] (-1.03,0)-- (-0.97,0);
\draw [line width=0.5pt] (-1.03,0.25)-- (-0.97,0.25);
\draw [line width=0.5pt] (-1.03,-1)-- (-0.97,-1);
\draw [line width=0.5pt] (-1.03,1)-- (-0.97,1);
\draw [line width=0.5pt] (-1,-1.03)-- (-1,-0.97);
\draw [line width=0.5pt] (-0.75,-1.03)-- (-0.75,-0.97);
\draw [line width=0.5pt] (-0.5,-1.03)-- (-0.5,-0.97);
\draw [line width=0.5pt] (-0.25,-1.03)-- (-0.25,-0.97);
\draw [line width=0.5pt] (0,-1.03)-- (0,-0.97);
\draw [line width=0.5pt] (1,-1.03)-- (1,-0.97);
\draw [line width=0.5pt] (0.75,-1.03)-- (0.75,-0.97);
\draw [line width=0.5pt] (0.5,-1.03)-- (0.5,-0.97);
\draw [line width=0.5pt] (0.25,-1.03)-- (0.25,-0.97);
\draw (-1.21,1.13) node[anchor=north west] {\tiny $1$};
\draw (-1.4,0.88) node[anchor=north west] {\tiny $3/4$};
\draw (-1.4,0.38) node[anchor=north west] {\tiny $1/4$};
\draw (-1.25,0.13) node[anchor=north west] {\tiny $0$};
\draw (-1.54,-0.12) node[anchor=north west] {\tiny $-1/4$};
\draw (-1.54,-0.62) node[anchor=north west] {\tiny $-3/4$};
\draw (-1.35,-0.87) node[anchor=north west] {\tiny $-1$};
\draw (-1.12,-1.37) node[rotate=90, anchor=north west] {\tiny $-1$};
\draw (-0.87,-1.55) node[rotate=90, anchor=north west] {\tiny $-3/4$};
\draw (-0.37,-1.55) node[rotate=90, anchor=north west] {\tiny $-1/4$};
\draw (-0.12,-1.25) node[rotate=90, anchor=north west] {\tiny $0$};
\draw (0.13,-1.4) node[rotate=90, anchor=north west] {\tiny $1/4$};
\draw (0.63,-1.4) node[rotate=90, anchor=north west] {\tiny $3/4$};
\draw (0.88,-1.25) node[rotate=90, anchor=north west] {\tiny $1$};
\end{tikzpicture}
\caption{}
\label{fig:counterexample3}
\end{minipage}\hfill
\begin{minipage}{0.33\textwidth}
\centering
\begin{tikzpicture}[line cap=round,line join=round,>=triangle 45,x=3cm,y=3cm]
\clip(-0.25,-0.2) rectangle (1.1,1.1);
\draw [line width=1pt] (0,1)-- (1,1);
\draw [line width=1pt] (1,1)-- (1,0);
\draw [line width=1pt] (1,0)-- (0,0);
\draw [line width=1pt] (0,0)-- (0,1);
\draw [line width=1pt] (0,0.5)-- (0.5,0.5);
\draw [line width=1pt] (0.5,0.5)-- (0.5,1);
\draw [line width=1pt] (0.5,0.5)-- (0.5,0);
\draw [line width=1pt] (0.25,1)-- (0.25,0.5);
\draw [line width=1pt] (0,0.75)-- (0.5,0.75);
\draw [line width=1pt] (0.5,0.5)-- (1,0.5);
\draw [line width=1pt] (0.125,1)-- (0.125,0.75);
\draw [line width=1pt] (0,0.875)-- (0.25,0.875);
\draw (-0.2,0.89) node[anchor=north west] {\tiny $2^{-3}$};
\draw (-0.2,1.03) node[anchor=north west] {\tiny $2^{-3}$};
\draw (-0.2,0.7) node[anchor=north west] {\tiny $2^{-2}$};
\draw (-0.2,0.37) node[anchor=north west] {\tiny $2^{-1}$};
\end{tikzpicture}
\caption{}
\label{fig:counterexample4}
\end{minipage}\hfill
\begin{minipage}{0.33\textwidth}
\centering
\begin{tikzpicture}[line cap=round,line join=round,>=triangle 45,x=1.7cm,y=1.7cm]
\clip(-1.1,-2.08) rectangle (1.1,0.2);
\draw [line width=1pt] (-1,-1.732)-- (0,0);
\draw [line width=1pt] (0,0)-- (1,-1.732);
\draw [line width=1pt] (-1,-1.732)-- (1,-1.732);
\draw [line width=1pt] (-1/2,-1.732/2)-- (0,-1.732);
\draw [line width=1pt] (1/2,-1.732/2)-- (0,-1.732);
\draw [line width=1pt] (-1/2,-1.732/2)-- (1/2,-1.732/2);
\draw [line width=1pt] (-1/4,-1.732/4)-- (0,-1.732/2);
\draw [line width=1pt] (1/4,-1.732/4)-- (0,-1.732/2);
\draw [line width=1pt] (-1/4,-1.732/4)-- (1/4,-1.732/4);
\draw [line width=1pt] (-1/8,-1.732/8)-- (0,-1.732/4);
\draw [line width=1pt] (1/8,-1.732/8)-- (0,-1.732/4);
\draw [line width=1pt] (-1/8,-1.732/8)-- (1/8,-1.732/8);
\end{tikzpicture}
\caption{}
\label{fig:counterexample5}
\end{minipage}\hfill
\end{figure}

The other results are sharp only in the sense that the exponential estimates are necessary. 
For Theorem~\ref{intro_tilebysquare}, this is witnessed
by the classical Fibonacci tiling: 
by induction a rectangle of size $u_n \times u_{n+1}$
can be tiled by $n$ squares of side-lengths 
$u_1,\ldots, u_n$, where $u_n$ is the $n$-th
Fibonacci number, which shows that 
$4^n$ cannot be replaced by anything
better than $u_{n+1}\approx ((\sqrt 5 + 1)/2)^{n+1} / \sqrt 5$ in Theorem~\ref{intro_tilebysquare}.
To get another construction, which is easier to generalize, we can trivially
tile a unit square, for any positive integer $k$, into $3$ squares of side length $2^{-i}$, for $i=1,..,k-1$, and $3+1$ squares of side length $2^{-k}$ (see Figure \ref{fig:counterexample4}). This way we have $n=3k+1$ tiling squares, and the side length of the smallest one is $2^{-k}=2^{-\frac{n-1}3}$.
By replacing squares by equilateral triangles 
in the previous construction (see Figure \ref{fig:counterexample5}), we
get a triangular tiling 
which 
shows that 
in Theorem~\ref{t:intro_strongt} one cannot replace $4^\frac{2n-2}{3}$ by anything better than $2^{\frac{n-1}3}$.
Finally, 
by replacing $3$ by $2^d-1$ and squares
by hypercubes in the square tiling, of which Figure \ref{fig:counterexample4} is an example, 
we get
a hypercube tiling which 
shows that in Theorem~\ref{h:intro_hypercubesbound} one cannot replace $4^{\frac{2(n-1)}{d}+1}$
by anything better than $2^{\frac{n-1}{2^d-1}}$. In terms of the classical setting of Prime Tiling, this example tiles a $d$-dimensional hypercuboid with side length $p=2^r$ by $r2^d=2^d\log_2(p)$ tiles, linear in $\log (p)$ as promised.

It is also worth noting, and easy to prove, that any axis-parallel rectangle tiling in the Cartesian plane has axis-parallel rectangular tiles. So, from now on we will always assume that these tiles have axis-parallel sides. Something similar happens when dealing with equilateral triangle tilings, as the sides of the tiles will also be parallel to the respective sides of the tiled figure.

The paper is organized as follows.
First, to illustrate and introduce some of our tools and techniques, 
we give a new simple proof of 
Max Dehn's theorem in Section~\ref{maxdehnsection} using simultaneous Diophantine approximation. 
In Section~\ref{s:almostinteger}, we prove Theorem~\ref{intro_keycorollary}.
In Section~\ref{s:numberofcoordinates}, we prove the estimates about the number of coordinates for rectangle tilings, hypercuboid tilings and for some triangle tilings
and 
we prove Theorem~\ref{intro_tilebysquare}.
Section~\ref{s:rectangles} contains our results about tilings with rectangles, Section~\ref{higherdim} is about the higher dimensional generalizations, and 
Section~\ref{s:triangles} is about tilings with triangles.



\section{A new proof of Max Dehn's theorem using
diophantine approximation}\label{maxdehnsection}

\begin{notation}
For 
$x\in \mathbb R$, we denote the distance between $x$ and the closest integer to $x$ by $\norm{x}$. 
\end{notation}

\begin{definition}\label{verthordef}
An important tool of this paper are 
the $(\frac12,\frac12)$-translates of the integer grid lines.
This grid is composed by lines of the form $\{k+\frac12\}\times \mathbb R\ (k\in \mathbb Z)$,
which we call \emph{$\frac12$-shifted vertical grid lines}, and by lines of the form $\mathbb R \times\{k+\frac12\}$,
which we call \emph{$\frac12$-shifted horizontal grid lines}.
{Here, we follow the convention that the $x$-axis is horizontal and is directed to the right, and the $y$-axis is vertical and directed upwards.}
\end{definition}

Max Dehn's theorem and Diophantine approximation are 
connected via the following two observations.

\begin{observation}\label{obs}
Let $R=[a,b]\times [c,d]$ be a rectangle tiled by
axis parallel squares. Suppose that
\begin{enumerate}[(i)]
\item
 at least one $\frac12$-shifted vertical grid line intersects $R$,
\item 
the $\frac12$-shifted vertical and horizontal grid lines do not contain
any vertex of any square tile,
\item \label{i:equal}
there are equal numbers of  $\frac12$-shifted vertical and  $\frac12$-shifted horizontal
lines through each square.
\end{enumerate}

Then the 
{aspect} ratio of $R$ is rational.
\end{observation}

\begin{proof}
Let $h$ and $v$ be the number of  $\frac12$-shifted horizontal and 
 $\frac12$-shifted vertical grid lines that intersect $R$, 
 respectively. 
By adding up the lengths of the intersections
of the  $\frac12$-shifted vertical grid lines and the squares, we get $v(d-c)$
and by adding up the lengths of the intersections
of the  $\frac12$-shifted horizontal grid lines and the squares, we get $h(b-a)$.
By \ref{i:equal} and because the two side lengths of a square are equal, 
this implies that $v(d-c)=h(b-a)$, and therefore 
$\frac{d-c}{b-a}=\frac{h}{v}\in\mathbb{Q}$. 
\end{proof}

\begin{observation}\label{keylemma}
If $\norm{a_i}<\frac{1}{4}$ for $i=1,2,3,4$ and $a_2-a_1=a_4-a_3 > 0$, then there are equal numbers of  $\frac12$-shifted vertical and  $\frac12$-shifted horizontal
lines through the square $[a_1,a_2]\times[a_3,a_4]$.
\end{observation}
\begin{proof}
As $\norm{a_i}<\frac{1}{4}$, we can write $a_i=A_i+\varepsilon_i$ uniquely for some $A_i\in \mathbb Z$ and $\varepsilon_i\in (-\frac14,\frac14)$. The number of $\frac12$-shifted horizontal and  $\frac12$-shifted vertical grid lines through the square are $A_4-A_3$ and $A_2-A_1$, 
respectively. 
Therefore, we only need to show $A_2-A_1=A_4-A_3$.

Since 
we have a square and $a_2-a_1=a_4-a_3$, 
\begin{align*}
    (A_2-A_1)+(\varepsilon_2-\varepsilon_1)&=(A_4-A_3)+(\varepsilon_4-\varepsilon_3)\\
    \implies (\varepsilon_2-\varepsilon_1)-(\varepsilon_4-\varepsilon_3)&=(A_4-A_3)-(A_2-A_1).
\end{align*}
However, the left-hand side is within $(-1,1)$ and the
right-hand side is an integer.
Hence, both sides of the equation are exactly $0$,
which completes the proof. 
\end{proof} 

\textit{Proof of Theorem~\ref{maxdehntheorem}.} Therefore, by the above observations, in order to 
prove that the  
aspect ratio of a rectangle with square tiling is rational all we need to do
is to re-scale the rectangle with the tiling such 
that both coordinates of every vertex of each square tile
have distance less than $\frac 14$ to integers
and the 
side-lengths of the rectangle 
are at least $1$. Such a scaling is provided by Theorem \ref{approximation}. $\hfill \Box$
\\ \\
The existence of such a scaling is clearly guaranteed by the
following well known theorem. To make the argument
self-contained and because the proof is short and elegant, 
we include it.

\begin{theorem}[Dirichlet's Approximation Theorem, Simultaneous Version]\label{approximation}
Given real numbers $\{a_1,...,a_k\}$ and a natural number $N>0$, there exists $q\in \mathbb N$, $1\leq q \leq N^k$ such that $\norm{qa_i}<\frac{1}{N}$ {for all $i$}. 
\end{theorem}
\begin{proof}
The proof is a simple application of the pigeonhole principle. 
Consider the cube $[0,1)^k\subset \mathbb R^k$,
which is the union of the $N^k$ disjoint cubes of the form $[\frac{i_1}{N},\frac{i_1+1}{N})\times[\frac{i_2}{N},\frac{i_2+1}{N}) \times \dots \times [\frac{i_k}{N},\frac{i_k+1}{N})$ for $(i_1,\dots,i_k)\in \{0,\dots,N-1\}^k$.
Take the $N^k+1$ points $(\{Ma_1\},\dots,\{Ma_k\})\in [0,1)^k$ for $0\leq M\leq N^k$ where $\{x\}$ denotes the fractional part of $x$.
By the pigeonhole principle, there are two points $(\{M_1a_1\},\dots,\{M_1a_k\})$ and $(\{M_2a_1\},\dots,\{M_2a_k\})$ with $0\leq M_1<M_2\leq N^k$ which are both elements of the same cube. 
Then taking $q=M_2-M_1$ suffices.
\end{proof}

\section{Square tilings with almost integer 
side-lengths must have integer side-lengths}
\label{s:almostinteger}

In this section, we exploit a very special property of the set of coordinates involved in a square tiling. Roughly speaking, if all coordinates of each square are close enough to integers, then that necessitates, with a few extra conditions, all coordinates to be exactly integers. 

The method used in this section, especially Lemma \ref{l:insteadofnetworks}, may seem to be pulled out of thin air. However, it is inspired by the correspondence between resistor networks and tilings of rectangles{: the side lengths of the tiles naturally produce a solution to Kirchoff's equations (see \cite{Prasolov}). 
We make an unexpected application of the uniqueness of this solution by using $\frac12$-shifted horizontal and vertical grid lines to produce another solution to obtain equalities. Now Lemma \ref{l:insteadofnetworks} is a proof of the uniqueness statement put into the tiling context, while {Lemma}~\ref{maintheorem} is its application.}
{We shall use Lemma~\ref{l:insteadofnetworks} for the function $r(x)=\lfloor x+1/2 \rfloor$, and in that case,  \ref{i:equal} of Observation~\ref{obs} will guarantee that \eqref{e:rcondition} of Lemma~\ref{l:insteadofnetworks} indeed holds.}
\begin{lemma} \label{l:insteadofnetworks}
Suppose that a rectangle $[a,b]\times[c,d]$ is tiled by
rectangles $R_i = [a_i,b_i]\times [c_i,d_i]$ $(i=1,\ldots,n)$
and we have a function $r:\mathbb{R}\to\mathbb{R}$
such that 
$r(c)=c$, $r(d)=d$ and
\begin{equation}\label{e:rcondition}
(b_i-a_i)\cdot(r(d_i)-r(c_i))=(r(b_i)-r(a_i))\cdot(d_i-c_i)
\qquad (i=1,\ldots,n).
\end{equation}

Then for every $i$ we have 
$$
r(c_i)=c_i, \quad
r(d_i)=d_i.
$$
\end{lemma}

\begin{proof}
Let $\delta(x)=r(x)-x$.
By subtracting $(b_i-a_i)\cdot (d_i-c_i)$ from 
\eqref{e:rcondition} we obtain
\begin{equation}\label{e:deltacondition}
(b_i-a_i)\cdot(\delta(d_i)-\delta(c_i))=
(\delta(b_i)-\delta(a_i))\cdot(d_i-c_i)
\qquad (i=1,\ldots,n).
\end{equation}

Take the union of the horizontal edges of the tiling rectangles and let $S_j = [u_j,v_j]\times\{h_j\}$ $(j=1,\ldots,m)$ be its connected components.
We can clearly suppose that $h_1=c$ and $h_m=d$.
Thus the assumptions $r(c)=c$ and $r(d)=d$ imply that 
\begin{equation}\label{e:h1hm}
\delta(h_1)=\delta(h_m)=0.
\end{equation}

Using \eqref{e:deltacondition} first and then 
\eqref{e:h1hm},
we obtain
\begin{align*}
&\sum_{i=1}^n \frac{b_i-a_i}{d_i-c_i}\cdot 
(\delta(d_i) - \delta(c_i))^2 =
\sum_{i=1}^n (\delta(b_i)-\delta(a_i))\cdot 
(\delta(d_i) - \delta(c_i))  \\
&= \sum_{j=2}^{m-1} \left(
  \sum_{i:d_i=h_j}(\delta(b_i)-\delta(a_i))
         \cdot \delta(h_j) -
  \sum_{i:c_i=h_j}(\delta(b_i)-\delta(a_i))
         \cdot \delta(h_j) \right).
\end{align*}
Since for any $j=2,\ldots,m-1$ 
the segment $[u_j,v_j]\times\{h_j\}$ can be obtained
both as the union of upper and lower edges of the
rectangles of the tiling,
we also have
$$
\delta(v_j)-\delta(u_j)=
\sum_{i:d_i=h_j}(\delta(b_i)-\delta(a_i))=
\sum_{i:c_i=h_j}(\delta(b_i)-\delta(a_i)).
$$
Thus the previous sum is zero, which implies that
$\delta(d_i)=\delta(c_i)$ for every $i$. 
{Notice that $\delta(c)=0$ and $\delta(d_i)=\sum_{j\in S}\left(\delta(d_j)-\delta(c_j)\right)+\delta(c)$ for some $S\subseteq \{1,\dots,n\}$ due to the connectivity in the tiling, hence $\delta(d_i)=\delta(c_i)=0$ for all $i$, which completes the proof.}
\end{proof}

By applying Lemma~\ref{l:insteadofnetworks} for
$r(x)=\left\lfloor x+\frac12 \right\rfloor$ 
(in other words, the
closest integer to $x$)
we get the following. 

\begin{lemma}\label{maintheorem}
Consider a rectangle tiled by squares. Suppose that the conditions (i)-(iii) of Observation~\ref{obs} are satisfied, and the top and bottom of the rectangle have integer $y$-coordinates. Then each square has side length equal to the number of 
 $\frac12$-shifted horizontal
lines through it. In particular, all squares have integer side lengths.
\end{lemma}

\begin{proof}
Let $r(x)=\left\lfloor x+\frac12 \right\rfloor$. 
We claim
that the conditions in Lemma \ref{l:insteadofnetworks} are satisfied. Indeed, as $c,d\in \mathbb Z$, $r(c)=c$ and $r(d)=d$. For each $i$, $b_i-a_i=c_i-d_i$ is the side length of the square; $r(b_i)-r(a_i)$ is the number of  $\frac12$-shifted vertical grid lines through the square, while $r(d_i)-r(c_i)$ is the number of  $\frac12$-shifted horizontal grid lines through the square. Hence $(c_i-d_i)(r(b_i)-r(a_i))=(b_i-a_i)(r(d_i)-r(c_i))$.
Applying Lemma~\ref{l:insteadofnetworks} we get
$d_i-c_i=r(d_i)-r(c_i)$ for each $i$, which completes the proof.
\end{proof}

\comment{\color{red} [I am (T) not sure that we need this paragraph.]
This may seem like both a bizarre and surprising theorem at first, and it indeed is. These seemingly weak and unrelated conditions infer such a strong 
{\color{blue} statement}
that implies, as an example, that there cannot be a square tile without vertical and horizontal grid lines through it.}

\begin{proof}[Proof of Theorem~\ref{intro_keycorollary}]
Combining Lemma~\ref{maintheorem} with Observation~\ref{keylemma}, we clearly obtain Theorem~\ref{intro_keycorollary}.
\end{proof}

\comment{\begin{corollary}\label{keycorollary}
{Suppose a rectangle tiled by squares in the Cartesian plane with sides parallel to the axes have a pair of integer-coordinated opposite sides. If each coordinate of every vertex of each square tile 
{\color{blue} has distance less than $\frac 14$ to} 
an integer, then all square tiles have integer side lengths.
}\end{corollary}}

\section{Bounding Number of Coordinate Values and the Proof of Theorem~\ref{intro_tilebysquare}}
\label{s:numberofcoordinates}

To use Theorem~\ref{intro_keycorollary} we need to re-scale the tiled rectangle, by using the Simultaneous Version of Dirichlet's Approximation Theorem (Theorem~\ref{approximation}), to make all coordinates of all tiling squares
closer than $\frac 14$ to
an integer.
How many coordinates do we have to work with?
If we have $n$ square tiles then, trivially, we cannot have more than $4n$ coordinates, and if we also notice that most coordinates are counted at least twice we get the immediate $2n+2$ estimate. 
To optimize the scaling factor as much as possible, we prove the following sharp estimate for the number of coordinates in a rectangle tiling. (To check that $n+3$ is indeed sharp is
left to the reader as an exercise.)

\begin{lemma}\label{l:coordvaluesforrectangles}
Suppose that a rectangle $R=[a,b]\times[c,d]$ is tiled by
rectangles $R_i=[a_i,b_i]\times [c_i,d_i]$ ($i=1,\ldots,n$).
Then the number of $x$-coordinates plus the number of $y$-coordinates of the vertices of the tiling rectangles $R_1,\ldots,R_n$ is at most $n+3$; that is,
\begin{equation}\label{e:ineqnplus3}
\# \{a_1,\ldots,a_n,b_1,\ldots,b_n\} +
\# \{c_1,\ldots,c_n,d_1,\ldots,d_n\} \le n+3.
\end{equation}
\end{lemma}

Now we show how Lemma~\ref{l:coordvaluesforrectangles} completes the proof of Theorem~\ref{intro_tilebysquare}, 
which we repeat here for convenience, and then we prove Lemma~\ref{l:coordvaluesforrectangles} in a more general form.

\begin{theorem} \label{defeatbyconway}
For any 
{tiling of a rectangle by $n$ squares}, one can find a scaling of it with longest side length at most $4^{n}$ such that all squares have integer side lengths. 
\end{theorem}
\textit{Proof {m}odulo Lemma~\ref{l:coordvaluesforrectangles}.}
We let the initial rectangle to have bottom left corner at coordinate $(0,0)$ and top left corner $(0,1)$, and let $(0,0)$ to $(0,1)$ be the longest side of the rectangle. Now consider the set
$$S:=\{x| (x,y) \text{ or } (y,x) \text{ is a vertex of a square tile for some }y\}\setminus \{0,1\}.$$
From Lemma~\ref{l:coordvaluesforrectangles}, $\abs{S}\leq (n+3)-3=n$, where we take away $3$ coordinates because we exclude $\{0,1\}$ 
{and in the left-hand side of \eqref{e:ineqnplus3} the number $0$ appears in both terms and $1$ appears in the second term since $(0,0)$ and $(0,1)$ are vertices of the rectangle.}

Applying Theorem \ref{approximation} on the set of real numbers $S$ with $N=4$, we can find $q\in \mathbb N$ such that $1\leq q\leq 4^{\abs{S}}$ and $\norm{qx}<\frac{1}{4}$ for all $x\in S$. Then if we scale the tiling by a factor of $q$, every vertex of each square has its 
coordinates in the form of $qx$ for some $x$ in $S$, and hence are 
closer than $\frac 14$ to an integer. Also, the rectangle has bottom side coordinate $0$ and top side coordinate $q$, which are both integers. Therefore by Theorem \ref{intro_keycorollary}, all squares have integer side length.

Now, the longest sides of the rectangle have size $q\leq 4^{\ \abs{S}}\leq 4^{n}$. $\hfill \Box$

\begin{remarks}
1. Without Lemma 4.1, using the trivial $2n+2$ estimate in \eqref{e:ineqnplus3} instead of $n+3$, we would get 
this theorem with the slightly worse $4^{2n-1}$ estimate instead of $4^n$.

2. Diophantine approximation is very carefully tailored to this situation. Not only does it very naturally guarantee the coordinates to be 
closer than $\frac 14$ to
integers, 
but the fact that $q$ is an integer also helps to 
ensure both $x$-coordinates or both $y$-coordinates of the vertices of the scaled rectangle are integers.
\end{remarks}

Instead of proving Lemma~\ref{l:coordvaluesforrectangles} directly,
we prove two generalizations of it. 
The first one (Lemma~\ref{l:JoseStrongEstimate})  will be needed for the 
higher dimensional variations, while the second one 
(Lemma~\ref{l:coordvaluesfortriangles})
for the variants with triangle tilings, and either of them clearly implies Lemma~\ref{l:coordvaluesforrectangles}.

\begin{lemma}\label{l:JoseStrongEstimate}
Suppose that a $d$-dimensional axis-parallel hypercuboid $C$ is tiled by $n$ hypercuboids $\{C_1,\dots,C_n\}$. Then
\begin{enumerate}[(i)]
    \item \label{i}
    the intersection of the interior of $C$ and the union
    of the boundaries of $C_1,\ldots,C_n$ can be covered by $n-1$ hyperplanes such that each hyperplane is perpendicular to one of the coordinate {axes}, and 

    \item \label{ii}
    there exist distinct $i,j\in\{1,\dots,d\}$
    such that the number of $i$-th coordinates plus
    the number of $j$-th coordinates of the vertices of the hypercuboids $\{C_1,\dots,C_n\}$ is at most
    $$ \frac{2(n-1)}{d}+4. $$
\end{enumerate}
\end{lemma}

\begin{proof}
It is easy to see that the \ref{i}
implies \ref{ii}, 
so it is enough to prove \ref{i}.
Fix one of the $d$ axis directions, say $i\in\{1,\dots, d\}$, and consider the union of the $2n$ $(d-1)$-dimensional faces of the $C_j$'s that are perpendicular to the given direction. Let $T_i$ be the collection of connected components of this union. 
Define $T_i'$ to be $T_i$, but without the top and bottom $(d-1)$-dimensional faces of $C$ and define $T=\bigsqcup_{i=1}^d T_i',$
{where here and in the sequel "top" indicates the largest $i$-th coordinate and "bottom" indicates the smallest $i$-th coordinate.}
Clearly, it is enough to prove that the set $T$ 
contains at most $n-1$ elements.

We consider the natural partial ordering on the points of $\mathbb{R}^d$, that is: we say that $p\preceq p'$ if $p_i\le p'_i$ for each coordinate of $p$ and $p'$.
Additionally, $p$ is a minimal point of $A\subset\mathbb{R}^d$
if $p\in A$ and 
{there does not exist $p'\in A$ such that $p'\prec p$.} Clearly, any compact subset of $\mathbb{R}^d$ has at least one minimal point and
any compact cuboid has a unique minimal point,
{which we call its bottom-left vertex}.

For each $S\in T$ we choose a minimal point $p$ of $S$.
Since $T$ does not contain any top face of $C$ this
point $p$ must be also a minimal point of a tiling cuboid $C_j$.
Also, as $T$ does not contain any bottom face of $C$, this $C_j$ cannot be the cuboid tile 
{that contains the bottom-left vertex}
of $C$.
Therefore, to prove that $T$ indeed has at most $n-1$ elements, and thus to complete the proof of the lemma,
it is enough to show that distinct $S, S'\in T$ cannot have a common minimal point.

So, suppose that $p=(p_1,\ldots,p_d)$ is a common minimal point of
these distinct $S, S'$. 
Choose $i$ and $i'$ such that $S\in T'_i$ and $S'\in T'_{i'}$.
Clearly, $i\neq i'$.
Due to the fact that $S\in T_i'$, we have that
for small enough $\varepsilon$ the open
{line segment}
$s$ between $p$ and $p-\varepsilon e_i$,
where $e_i$ denotes the unit vector in the direction of the $i$-th axis, is contained in $C_j$ for some $j$, see Figure~\ref{fig:cross-section}.
As $s$ is contained in the hyperplane of $S'$,
the minimality of $p$ in $S'$ implies that $s$ must be in the interior of $C_j$.
Furthermore, this implies that if we write $C_j=[u_1,v_1]\times \ldots \times [u_d,v_d]$, then we must have $v_i=p_i$.
But then the top $(d-1)$-dimensional face of $C_j$
perpendicular to the $i$-th axis must be 
contained in $S$ and so $p$ is not minimal in
$S$, which is a contradiction.

\begin{figure}[ht]
\begin{minipage}{0.4\textwidth}
\captionsetup{width=1\textwidth}
\centering
\begin{tikzpicture}[line cap=round,line join=round,>=triangle 45,x=1.5cm,y=1.5cm]

\draw [-stealth] [line width=1pt] (-1.5,-1.55)-- (-1.5,1.5);
\draw [-stealth] [line width=1pt] (-1.55,-1.5)-- (1.5,-1.5);
\draw [line width=1pt] (-1.25,-1.25) rectangle (1.25,1.25);
\draw [line width=1pt,dashed] (-1.0,-1.0) rectangle (-0.25,0.25);
\draw [line width=0.5pt] (-0.6,-0.6) -- (-0.25,-0.6);
\draw [line width=1pt] (-0.25,-0.6) -- (0.75,-0.6);
\draw [line width=1pt] (-0.25,-0.6) -- (-0.25,1.0);
\draw (1.25,1.25) node[anchor=north west] {\tiny $C$};
\draw (-0.9,-0.7) node[anchor=north west] {\tiny $C_j$};
\draw (-0.25,1.0) node[anchor=north west] {\tiny $S$};
\draw (0.6,-0.3) node[anchor=north west] {\tiny $S'$};
\draw (-0.55,-0.4) node[anchor=north west] {\tiny $s$};
\fill (-0.25,-0.6) circle(.04);
\fill (-0.6,-0.6) circle(.04);
\draw (-0.25,-0.3) node[anchor=north west] {\tiny $p$};
\draw (-1.5,1.5) node[anchor=north east] {\tiny $i'$};
\draw (1.5,-1.55) node[anchor=north east] {\tiny $i$};

\end{tikzpicture}
\caption{Cross section through $p$}
\label{fig:cross-section}
\end{minipage}\hfill
\end{figure}
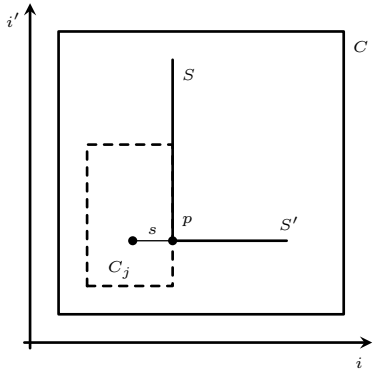
\end{proof}

\begin{lemma}
\label{l:coordvaluesfortriangles}
Suppose that $T$ is 
tiled by tiles $T_1,\ldots,T_n$,
where $T$ and each $T_i$ is a trapezoid with two horizontal sides or a triangle with a horizontal side. 
Then the intersection of the interior of $T$ and the union $H$ of the boundaries of $T_1,\ldots,T_n$ can be covered by $n-1$ lines. 
\end{lemma}

\begin{proof}
Let $L$ be the set of lines that contains at least one of the line segments of $H$. 
We need to show that $\# L\le n-1$. 
To this end, we define an injective, but not surjective, function $f:L\to\{T_1,\ldots,T_n\}$.

If $\ell\in L$ is horizontal, then consider those tiles
whose boundaries intersect $\ell$ and are contained in the closed half-plane above $\ell$.
It is easy to see that there exists at least one such tile and there is a clear ordering of these tiles from left to right, so, we can define $f(\ell)$ as the the leftmost one. 
If $\ell\in L$ is non-horizontal then consider the tiles whose left edge is on $\ell$ and let $f(\ell)$ be the one with the lowest left side. 
Note that the bottom-left tile is not in the range of $f$ since its domain $L$ does not contain the left and bottom sides of $T$. 
Thus $f$ is not surjective.

To prove that $f$ is injective it is enough to check that we cannot have $f(h)=T_i=f(v)$ for any horizontal $h\in L$ and nonhorizontal $v\in L$. 
Suppose, to the contrary, that this is the case.
Let $(a,b)=h\cap v$. 
This point must be the bottom-left vertex of $T_i$.
{Now observe that $v$ is not the left side of $T$ because $H$ is contained in the interior of $T$, so no line segment of $H$ is contained in the boundary of $T$. Thus,} for small enough $\varepsilon>0$ the line segment $s=(a-\varepsilon,a)\times\{b\}$ must be contained in some tile $T_j$.
The segment $s$ must be in the interior of $T_j$, 
otherwise $T_i$ would not be the leftmost tile whose boundary intersects $h$ and is contained in the closed half-plane above $h$. 
Using that $T_j$ is a trapezoid with two horizontal sides or a triangle with a horizontal side, this implies that $(a,b)$ is in the interior of the right edge of $T_j$.
Since $T_i$ is the leftmost tile whose boundary intersects $h$ and is contained in the closed half-line above $h$, there cannot be a tile between the right edge of $T_j$ and the left edge of $T_i$, so $v$ contains the the right edge of $T_j$.
But then, on the other side of $v$ there must be at least one tile lower than $T_i$, contradicting $f(v)=T_i$.
\end{proof}

\section{Tiling by Rectangles}
\label{s:rectangles}




In this section, we consider tiling rectangles with smaller rectangles of pre-emptively fixed sizes 
and we prove Theorem~\ref{intro_scaleqrect} mentioned in the introduction. We work under the set up that a rectangle is tiled by $n$ smaller rectangular tiles, where the $i$-th tile has ratio-of-sides $p_i/q_i$, where $p_i\in \mathbb{N}^{+}$ corresponds to the vertical side and $q_i\in \mathbb{N}^{+}$ to the horizontal side. We investigate whether there will be similar results as the case of square tiles. 
\comment{A related problem was investigated first in \cite{Fomin}, where it is given an upper bound to the maximal size of the tiled rectangle, in terms of $\max \{p_i, q_i\} /\gcd(p_i, q_i)$, called there as the bias, and the greatest common divisor of all the sizes of the tiles. As stated there, that result is a generalization of Theorem 4 of Conway's paper \cite{Conway}.
}


In the same manner as for squares, we get the following 
more general form of Lemma~\ref{maintheorem} by 
again applying Lemma~\ref{l:insteadofnetworks} for 
$r(x)=\lfloor x+\frac12 \rfloor$.

\begin{theorem}\label{rectangletheorem}
Consider a rectangle 
tiled by $n$ rectangles.
The $i$-th rectangle, $R_i$, has ratio {of its vertical side to its horizontal side equal to $p_i/q_i$ where $p_i,q_i\in \mathbb{N}^{+}$.} 
Suppose that
\begin{itemize}
    \item the {top and bottom} {vertices} of the tiled rectangle have integer $y$-coordinates;
    \item there is no {side} {vertex} of any rectangular tile having $y$-coordinate with fractional part $\frac 12$; and
    \item the ratio of the number of $\frac12$-shifted horizontal and $\frac12$-shifted vertical grid lines through the $i$th tile is $p_i/q_i$.
\end{itemize}
Then each rectangle has horizontal side length equal to the number of $\frac12$-shifted vertical grid lines through it.
\end{theorem}


Using a slightly different Diophantine approximation from the square case, we arrive at the following analogue of Theorem {\ref{intro_keycorollary}}.

\begin{lemma}\label{rectanglelemma}
Given a 
{tiling of a rectangle by $n$ rectangles}
with the ratios of the vertical side to the horizontal side equal to $\frac{p_i}{q_i}$ with $p_i,q_i\in \mathbb{N}^{+}$. Suppose the $i$th rectangle $R_i$ has actual size $(a_i:b_i)$, and 
\begin{itemize}
    \item the 
    vertices of the 
    {tiled} rectangle have integer $y$-coordinates;
    \item for each $i$, the ratio $a_i/p_i$ is 
    {at most distance
    $\frac{1}{{2}\max\{p_i,q_i\}}$ away from an integer;}
    \item all coordinates of each vertex of the tiles are 
    closer than $\frac 14$ to
    an integer.
\end{itemize}
Then {the} {side lengths of the $i$th rectangular tile are integer multiples of $(p_i,q_i)$. In particular, }the side lengths of all the rectangular tiles are integers.
\end{lemma}

\begin{proof}
Define
\begin{equation*}
    c_i = \frac{a_i}{p_i} = \frac{b_i}{q_i}.
\end{equation*}
Now consider the $i$th rectangle tile $R_i$, in particular 
the $y$-coordinates of its vertices. Write these two coordinates to be $r_1+\varepsilon_1<r_2+\varepsilon_2$, where $r_j\in \mathbb Z$ and $\varepsilon_j\in (-\frac {1}{{4}},+\frac {1}{{4}})$. Moreover let $c_i=s_i+\delta_i$ where $s_i\in\mathbb Z$ and $\delta_i\in (-\frac{1}{{2}\max\{p_i,q_i\}},+\frac{1}{{2}\max\{p_i,q_i\}})$ by the 
hypotheses of the lemma. Now we have
\begin{align*}
    (r_2+\varepsilon_2)-(r_1+\varepsilon_1)=a_i&=c_ip_i=(s_i+\delta_i)p_i\\
    \implies r_2-r_1-s_ip_i&=\delta_ip_i-\varepsilon_2+\varepsilon_1
\end{align*}
Now $LHS \in \mathbb Z$ and $RHS \in (-1,1)$. Hence both sides are exactly $0$. In particular $r_2-r_1=s_ip_i$. In other words, the number of $\frac12$-shifted horizontal grid lines through the tile $R_i$ is $s_ip_i$. One could similarly show that the number of $\frac12$-shifted vertical grid lines through $R_i$ is $s_iq_i$.

We are now ready to apply Theorem \ref{rectangletheorem} to this tiling because all 
three conditions are satisfied.
Therefore each rectangle tile $R_i$ has horizontal side length equal to the number of $\frac12$-shifted vertical grid lines through it, $s_iq_i$. 
Then clearly the vertical side length is $s_ip_i$.
\end{proof}

{One can notice that Lemma \ref{rectanglelemma} is exactly Theorem \ref{intro_keycorollary} in the case where tiles have 
dimension $(1:1)$, that is, they
{are squares.}

\begin{theorem}[Dirichlet's Approximation Theorem, Varied Simultaneous Version]\label{varieddirichlet} Given sets of real numbers $S_0,\dots,S_r$ and natural numbers $m_0,\dots,m_r$, there exists $q\in \mathbb{N}^{+}$, with $1\leq q\leq m_0^{\abs{S_1}}\dots m_r^{\abs{S_r}}$ such that $\norm{qa_i}<\frac{1}{m_i}$ for any $a_i\in S_i$ and for all $i$.
\end{theorem}

\begin{proof}
This can be proved analogously to Theorem \ref{approximation}.
\end{proof}

We 
now have the conditions to prove Theorem~\ref{intro_scaleqrect},
which we repeat here for the readers' comfort.

\begin{theorem}\label{scaleqrect}
Given a 
{tiling of a rectangle by $n$ rectangles} with ratios of side lengths $p_i/q_i$, $p_i,q_i\in \mathbb{N}^{+}$ respectively, one can find a scaling of it with longest side length at most $8^{n}\prod_{i=1}^n\max\{p_i,q_i\}$ such that all rectangles have integer side lengths.
\end{theorem}

\begin{proof}
Again, we place the rectangle such that the bottom left corner 
has coordinate $(0,0)$, top left corner $(0,1)$ and let $(0,0)$ to $(0,1)$ be the longest side of the rectangle. For the $i$th rectangle tile $R_i$, denote the vertical side-length by $a_i$ and the horizontal one as $b_i$. 
Define
\begin{equation*}
    c_i = \frac{a_i}{p_i} = \frac{b_i}{q_i}.
\end{equation*}
Let 
\begin{align*}
    S_0&:=\{x\mid  (x,y) \text{ or } (y,x) \text{ is a vertex of a rectangle tile for some }y\}\setminus \{0,1\},\\
    S_i&:=\{c_i\} \qquad (i=1,\dots,n).
\end{align*}
Now we apply Theorem \ref{varieddirichlet} to these sets, $m_0 = 4$ and $m_i = 2\max\{p_i,q_i\}$, to find $q\in \mathbb N$, $1\leq q \leq {4}^{\abs{S_0}}\prod_{i=1}^n\left(({2}\max\{p_i,q_i\})^{\abs{S_i}}\right)$ such that 

\begin{equation*}
  \norm{qc_i} < \frac{1}{{2} \max\{p_i, q_i\}}, \text{ and } \norm{qx}< \frac{1}{{4}}\text{ for all }x\in S_0. 
\end{equation*}

Scale the entire tiling by $q$. Now apply Lemma \ref{rectanglelemma} to the scaled tiling to obtain that all sides of all rectangles are of integer length.

The longest side of the rectangle is $q$, and we know $q\leq {4}^{\abs{S_0}}\prod_{i=1}^n({2}\max \{p_i,q_i\})$. By Lemma~\ref{l:coordvaluesforrectangles}, $\abs{S_0}\leq (n+3)-3=n$. Therefore, $q\leq {8}^{n}\prod_{i=1}^n\max\{p_i,q_i\}$.
\end{proof}

\begin{remark}
An analog of Lemma \ref{keylemma} does exist in this case, but the resulting bounds are worse. Therefore, we have introduced $c_i$'s to guarantee the number of $\frac12$-shifted vertical and $\frac12$-shifted horizontal grid lines are in desired proportions after scaling without the help of an analogous lemma. 
\end{remark}

\comment{
The equivalent result obtained in \cite{Fomin} gives the upper bound to the longest side of $d \cdot \tau^{\sum^n_{i=1} \frac{\max(p_i,q_i)}{\gcd(p_i, q_i)}}$, where $\tau \approx 1.8637$ is the constant presented in Theorem 2 of \cite{Fomin}, and $d$ is the greatest common divisor of all the rectangular tiles. 
}


\section{Tiling of Hypercuboids}\label{higherdim}
\comment{{
In this section, we study the case for tiling a hypercuboid with hypercubes/hypercuboids. For a long time, this higher dimensional analogue of the same easily answered planar problem has been known to be difficult to attack. As we mentioned in the introduction, both the end of \cite{BSST} and \cite{Kenyon} explained that the main reason for this is the lack of a generalization of the resistor network technique, which seems to have caused a dead-end to this trail of thought. Our following result treats this analogue problem and improves the bound stated in \cite{Walters}, via Diophantine Approximation, which we talked about in the presentation of Theorem \ref{h:intro_hypercubesbound}.
}}
We start this subsection by proving Theorem~\ref{h:intro_hypercubesbound}. As before, we write it again for convenience.

\begin{theorem}\label{h:hypercubesbound}
For any 
{tiling of a hypercuboid with $n$ hypercubes}, one can find a scaling of it with 
shortest side length at most $4^{\frac{2(n-1)}{d}+1}$
and longest side length at most $n\cdot 4^{\frac{2(n-1)}{d}+1}$ such that all tiling hypercubes have integer side lengths.
\end{theorem}

\begin{proof}
Let $C=\{C_1,\dots, C_n\}$ be the set of tiles. 
By Lemma \ref{l:JoseStrongEstimate} we know that there exist 
two distinct coordinates, say the $i$-th and $j$-th coordinates, such that the number of $i$-th coordinates plus the number of $j$-th coordinates of the vertices of the tiles is at most $\frac{2(n-1)}{d}+4$. Without loss of generality we can assume that $i=1$ and $j=2$.

Now, we translate and scale the hypercuboid in question, $H$, so that it has
vertices at $(0,\dots, 0)$ and $(1, 0, \dots, 0)$.

Let $S$ be the set of the first and second coordinates of all the vertices. Noticing 
two of the four 
first and second coordinates of the vertices of $H$ are both $0$, we can conclude $S \setminus \{0,1\}$ has at most $\frac{2(n-1)}{d}+1$ elements. Therefore, applying Theorem \ref{approximation} on $S\setminus \{0,1\}$ with $N=4$, we find $1\leq q \leq 4^{\frac{2(n-1)}{d}+1}$, $q \in \mathbb{N}^{+}$ such that $\norm{q x} <\frac 14$ for all $x\in S$.

First, scale $H$ by $q$. Then let $P$ be any arbitrary $2$-dimensional plane parallel to the first and second axes that does not pass through any vertices of the tiles, and whose intersection with $H$ is nonempty. Consider the rectangle $R$ dissected by the collection of nonempty squares among $C_i \cap P$’s. Applying Theorem~\ref{intro_keycorollary} to this rectangle 
tiling
implies all hypercubes involved in the section have integer side lengths. However, the choice of the section is arbitrary, so all hypercubes must have integer side lengths. The first coordinate side has length $q \leq 4^{\frac{2(n-1)}{d}+1}$. To finish, note that every side of $H$ is partitioned by sides of the tiles in $C$, so the longest side of $H$ is at most $n$ times the smallest side of $H$. Since we know that the first coordinate side of $H$ is at most $4^{\frac{2(n-1)}{d}+1}$, we can conclude that the shortest side length of $H$ is at most $4^{\frac{2(n-1)}{d}+1}$
and the longest side is 
at most $n\cdot 4^{\frac{2(n-1)}{d}+1}$.
\end{proof}

\begin{remark}
The estimate $4^n$ of Theorem~\ref{intro_tilebysquare} also works here: enumerating the coordinates so that the longest side of the hypercuboid is in the first coordinate, and using \ref{i} of Lemma \ref{l:JoseStrongEstimate}, we have that the number of the coordinates in any two directions is at most $n+3$. A similar argument as in Theorem \ref{h:hypercubesbound} gives a scaling for any hypercuboid tiled by $n$ hypercubes, with longest side length at most $4^n$, such that all hypercubes have integer length. On the other hand, if $d>2$, 
the $4^n$ bound is a better estimate than the one in Theorem \ref{h:hypercubesbound} only for $n=2$ and $3$, in most cases it is much weaker.
\end{remark} 

Now we prove a higher dimensional generalization of
Theorem~\ref{intro_scaleqrect}.

\begin{theorem}
Suppose there is a 
{tiling of a hypercuboid by $n$ hypercuboids} where the $k$-th tile
is similar to a hypercuboid 
$q_{k,1} \times \dots \times q_{k,d}$, where $q_{k,{j}}\in \mathbb{N}^{+}$ {for all $j=1,\dots,d$}. 
{Let the longest side of the hypercuboid be along the $i$-th coordinate.}
Then{, for any $j\neq i$}
one can find a scaling of 
{the tiled hypercuboid}
with longest side length at most ${8}^{n}\prod_{k=1}^n\max\{q_{k,i},q_{k,j}\}$ such that all 
{tiles}
have integer side lengths.
\end{theorem}
\begin{proof}

Let $C=\{C_1,\dots, C_n\}$ be the set of tiles. 
Without loss of generality the side $(0,\dots,0)$ to $(1,0,\dots,0)$ is the longest of the hypercuboid,
{in particular, $i=1$}.
Take any direction $j\neq 1$ and for the $k$th hypercuboid tile $C_k$, denote the first axis direction side length by $a_k$ and the $j$th one as $b_k$. Then we 
have $\frac{a_k}{b_k}=\frac{q_{k,1}}{q_{k,j}}$. Define
\begin{equation*}
    c_k = \frac{a_k}{q_{k,1}} = \frac{b_k}{q_{k,j}}.
\end{equation*}
Let 
\begin{align*}
    S_0&:=\text{Collection of all first and $j$-th coordinates of all vertices of the tiles}\setminus \{0,1\},\\
    S_k&:=\{c_k\} \qquad (k=1,\dots,n).
\end{align*}
Now we apply Theorem \ref{varieddirichlet} to these sets, to find $q\in \mathbb N$, $1\leq q \leq {4}^{\abs{S_0}}\prod_{k=1}^n\left(({2}\max\{q_{k,1},q_{k,j}\})^{\abs{S_k}}\right)$ such that 

\begin{equation*}
  \norm{qc_k} < \frac{1}{{2} \max\{q_{k,1}, q_{k,j}\}}, \text{ and } \norm{qx}< \frac{1}{{4}}\text{ for all }x\in S_0. 
\end{equation*}

Scale the entire tiling by $q$
{and let $H$ be the tiled hypercuboid we obtain}. 
As in Theorem \ref{h:hypercubesbound}, we take $P$, any arbitrary $2$-dimensional plane parallel to the first and $j$-th axes that does not {pass} 
through any vertices, and whose intersection with $H$ is nonempty. Consider the rectangle $R$ dissected by the collection of nonempty rectangles among $C_i \cap P$’s. 
Then $R$ together with its tiles satisfy the conditions of Lemma \ref{rectanglelemma}. {Hence we conclude that the ordered pair of side lengths of $C_i \cap P$ is an integer multiple of $(q_{i,1},q_{i,j})$. Therefore, all the side lengths of the hypercuboids are integers.} 

The longest side of $H$ is $q$, and we know $q\leq {4}^{\abs{S_0}}\prod_{k=1}^n({2}\max\{q_{k,1},q_{k,j}\})$. By Lemma \ref{l:JoseStrongEstimate}, $\abs{S_0}\leq n$. Therefore $q\leq {8}^{n}\prod_{k=1}^n\max\{q_{k,1},q_{k,j}\}$.
\end{proof}

We could also have given an analogous bound to the longest side of $H$ in terms of the smallest side, as in the final part of the proof of Theorem \ref{h:hypercubesbound}, but in this case the $d$, possibly different, side lengths of each tile would have given a more complex and artificial factor, rather than the elegant $n$ factor found in Theorem \ref{h:hypercubesbound}. We left to the reader the task 
of obtaining this bound.

\section{Triangle tilings}
\label{s:triangles}

In this section, we study the dissections of an equilateral triangle or a trapezoid into equilateral triangles 
and we prove results similar to the ones obtained in previous sections.
Here it will be more convenient to use an oblique coordinate system with angle $60^\circ$ between the axes. 
We use $(1,0)$ and $(1/2,\sqrt{3}/2)$ for base vectors, and we will denote by $\Psi(a,b)$ the point with coordinates $a$ and $b$ in this oblique coordinate system; that is, we let
$$
\Psi(a,b)= a(1,0) + b\left(\frac{1}{2},\frac{\sqrt3}{2}\right)=
\left(a+\frac{b}{2}, \frac{\sqrt3 b}{2}\right).
$$
We will refer from now on to the coordinates of this system as the ``$\Psi$-coordinates''. Let $T(a,b,c)$ denote the equilateral triangles with vertices of the form $\Psi(a,b),\Psi(a+c,b),\Psi(a,b+c)$,
where $a,b,c\in\mathbb{R}, c\neq 0$. 
(Note that $c$ can be negative.).


The proof of the following lemma was partly inspired by  \cite{BSST75} in which ``leaky'' electric flows are assigned to triangle tilings of triangles and a uniqueness result is proved. 

\begin{lemma}
Suppose that an equilateral 
triangle $T(A,B,C)$ or a trapezoid with parallel sides on the horizontal lines $\Psi(\mathbb{R}\times\{B\})$ and $\Psi(\mathbb{R}\times\{B+C\})$
is tiled by triangles $T(a_i,b_i,c_i)$ ($i=1,\ldots,n$)
and we have a function $r:\mathbb{R}\to\mathbb{R}$
such that $r(B)=B$, $r(B+C)=B+C$ and
\begin{equation}\label{e:trcondition}
r(a_i+c_i)-r(a_i)=r(b_i+c_i)-r(b_i)
\qquad (i=1,\ldots,n).
\end{equation}

Then for every $i$ we have 
$$
r(b_i)=b_i, \quad
r(b_i+c_i)=b_i+c_i.
$$
\end{lemma}

\begin{proof}
Let $\delta(x)=r(x)-x$.
By  
\eqref{e:trcondition} we also have
\begin{equation}\label{e:tdeltacondition}
\delta(a_i+c_i)-\delta(a_i)=
\delta(b_i+c_i)-\delta(b_i)
\qquad (i=1,\ldots,n).
\end{equation}

Let $S_j=\Psi([u_j,v_j]\times\{h_j\})$ $(j=1,\ldots,m)$ be
the connected components of the union of the horizontal
edges and the vertices of the triangles of the tiling.
We can clearly suppose that $B=h_1 \leq \dots \leq h_m=B+C$.
Thus the assumptions $r(B)=B, r(B+C)=B+C$ imply that 
\begin{equation}\label{e:th1hm}
\delta(h_1)=\delta(h_m)=0.
\end{equation}

Fix $j\in\{2,\ldots,m-1\}$ and consider the line segment $S_j=\Psi([u_j,v_j]\times\{h_j\})$.
It contains vertices of the form $\Psi(a_i,b_i+c_i)$ 
(such that $b_i+c_i=h_j$) and
horizontal sides of the form $\Psi([a_i,a_i+c_i]\times \{b_i\})$
or $\Psi([a_i+c_i,a_i]\times \{b_i\})$
(such that $b_i=h_j$) 
of some tiling triangles.
{Consider the triangles having one side contained in $S_j$. There are two options for a triangle, either it is completely located in the upper-half plane determined by the extension of $S_j$ (that is of the form $\Psi([a_i,a_i+c_i]\times \{h_j\})$ with $c_i > 0$), or located in the lower-half plane of this extension (that is of the form $\Psi([a_i+c_i,a_i]\times \{h_j\})$ with $c_i < 0$). 
{Each of these two sets gives}
a partition for $S_j$, so}
\begin{align}\label{e:tdiffsum}
    \sum_{i:b_i=h_j, c_i > 0}(\delta(a_i+c_i)-\delta(a_i))= \delta(v_j)-\delta(u_j) &= -\sum_{i:b_i=h_j, c_i < 0}(\delta(a_i+c_i)-\delta(a_i)) \implies \nonumber  \\
    \sum_{i:b_i=h_j}(\delta(a_i+c_i)-&\delta(a_i))=0.
\end{align}

Since the vertices and horizontal sides are alternating as we 
go around the segment $S_j$, the numbers of vertices and sides bordering upon the fixed $S_j$ ($j=2,\ldots,m-1$) are
equal. {So, considering the horizontal sides of $T(a_i,b_i,c_i)$ and the vertices $\Psi(a_i,b_i+c_i)$ in a fixed $S_j$ we have}
\begin{align*}
    \sum_{i:b_i=h_j} \delta^2(b_i) = \delta^2(h_j) \cdot \#\{\text{Sides in } S_j\} = \delta^2(h_j) \cdot \#\{\text{Vertices in } S_j\} = \sum_{i:b_i+c_i=h_j} \delta^2(b_i+c_i).
\end{align*}

Thus, summing for all $j \in \{1,\dots,m\}$ and recalling that $\delta(h_1)=\delta(h_m)=0$ we get
\begin{equation}\label{e:tcsum}
\sum_{i=1}^n \delta^2(b_i+c_i) =
\sum_{i=1}^n \delta^2(b_i).
\end{equation}

Using \eqref{e:tcsum},  
\eqref{e:tdeltacondition}, \eqref{e:th1hm}
and \eqref{e:tdiffsum}
we obtain
\begin{align*}
\sum_{i=1}^n  
(\delta(b_i+c_i) - \delta(b_i))^2 
&= \sum_{i=1}^n  
(\delta^2(b_i+c_i) + \delta^2(b_i)-
2\delta(b_i+c_i)\delta(b_i)) \\
&= \sum_{i=1}^n  
(2\delta(b_i)^2-
2\delta(b_i+c_i)\delta(b_i)) \\
&= \sum_{i=1}^n  
-2\delta(b_i)(\delta(b_i+c_i)-
\delta(b_i)) \\
&= \sum_{i=1}^n  
-2\delta(b_i)(\delta(a_i+c_i)-
\delta(a_i)) \\
&= \sum_{j=2}^{m-1} -2\delta(h_j)
\sum_{i:b_i=h_j}(\delta(a_i+c_i)-\delta(a_i))=0.
\end{align*}

Therefore $\delta(b_i+c_i)=\delta(b_i)$ for every $i$.
{
Consider the following graph. Let $V=\{b_1,\ldots,b_n,b_1+c_1,\ldots,b_n+c_n\}$ be the vertex set and connect the numbers $x$ and $y$ by an edge if there exists $i$ such that $x=b_i, y=b_i+c_i$. Since $\delta(b_i+c_i)=\delta(b_i)$ for every $i$, the function $\delta$ is constant on each connected component of the graph. 
On the other hand, by construction, 
$\min V = B$ and every $x\in V\setminus\{B\}$ has an adjacent vertex $y\in V$ such that
$y<x$. 
By finiteness, repeating this we obtain a path
from any vertex of the graph to $B$, which implies that the graph is connected, so $\delta(b_i+c_i)=\delta(b_i)=\delta(B)$ for every $i$.
Since by assumption $\delta(B)=0$,
this completes the proof.}
\end{proof}

Before going to the next result notice that, unlike the previous sections, here we will not make use of the ``$\frac12$-shifted vertical grid lines'', presented in Definition \ref{verthordef}. Instead, we will use the lines resulting after applying $\Psi$ to them, that is of the form $\Psi(\{ y + 1/2 \} \times \mathbb{R})$, and, henceforth, we will denote them by ``$\frac12$-shifted $60^\circ$ diagonal grid lines''. The $\frac12$-shifted horizontal grid lines remain unchanged by $\Psi$, so it is not necessary to rename them. Now, for $r(x)=\left\lfloor x+\frac12 \right\rfloor$ 
the previous lemma immediately gives the next result, analogous to Lemma \ref{maintheorem}.

\begin{theorem}\label{triangtheorem}
Consider an equilateral triangle, $T$, 
with two of its sides parallel to the basis vectors $\Psi(1,0)$ and $\Psi(0,1)$, tiled by equilateral triangles. Suppose that
\begin{itemize}
    \item the top and bottom of $T$ have integer second $\Psi$-coordinates;
    \item there are no vertices of any triangle tile having $\Psi$-coordinates with fractional part $\frac 12$; and
    \item there is an equal number of $\frac12$-shifted $60^\circ$ diagonal and $\frac12$-shifted horizontal grid lines through the sides of each triangle tile.
\end{itemize}
Then each triangle tile has side length equal to the number of $\frac12$-shifted diagonal grid lines through it. In particular, all equilateral triangles have integer side length.
\end{theorem}

We next present the analogous version of Observation~\ref{keylemma}.

\begin{observation}\label{obstriangle}
If $\norm{a_i}<\frac{1}{4}$ for $i=1,2,3,4$ and $l=a_2-a_1=a_4-a_3 > 0$, then there are equal numbers of  $\frac12$-shifted $60^\circ$ diagonal and $\frac12$-shifted horizontal grid lines through the triangle  $T(a_1,a_3,l)$.
\end{observation}

Combining Theorem~\ref{triangtheorem} with Observation~\ref{obstriangle}, we obtain the following result.
\begin{corollary}\label{triangcorollary}
Suppose an equilateral triangle
with two of its sides parallel to the basis vectors $\Psi(1,0)$ and $\Psi(0,1)$ is tiled by equilateral triangles.
If all $\Psi$-coordinates of every vertex of each triangular tile are 
closer than $\frac 14$ to
an integer, then all triangular tiles have integer side lengths.
\end{corollary}

Next, we provide a proof for Theorem~\ref{t:intro_strongt},
which we state again for convenience.
In contrast with the rectangular case, in a triangular tiling by equilateral triangles, we need a few more steps, as we will see, before we can obtain a scaling of it in which all of its tiles have integer side lengths.

\begin{theorem}\label{t:strongt}
For any 
{tiling of an equilateral triangle by $n$ equilateral triangles}, one can find a scaling of it with side length at most $4^{\frac{2n-2}{3}}$ such that all triangular tiles have integer side lengths.
\end{theorem}

\begin{proof}
Let $T$ be our 
tiled triangle. Without loss of generality we can assume that $T$ is $T(0,0,1)$. By Lemma \ref{l:coordvaluesfortriangles}, we know that there are at most $n-1$ lines (which are parallel to one of the three following vectors: $\Psi(0,1)$, $\Psi(1,0)$ and $\Psi(-1,1)$) containing the union of boundaries of the tiles in the interior of $T$. Considering the $0$°, $120$° and $240$° rotations of $T$ and the triangles tiling it, with respect to the center of $T$, we can assure, by the Pigeonhole Principle, that in one of these three 
configurations the total number of first and second $\Psi$-coordinates is less than or equal to $\frac{2n-2}{3}$. Let $S$ be this set of $\Psi$-coordinates. We then replace the tiling with the corresponding tiling resulting from the rotation where this happens. Since the rotation is with respect to the center of $T$, the triangle $T$ remains unchanged. Applying Theorem \ref{approximation} on the set of real numbers $S$ with $N=4$, we can find $q\in \mathbb N$ such that $1\leq q\leq 4^{\abs{S}}
\le 4^\frac{2n-2}3$ and $\norm{qx}<\frac{1}{4}$ for all $x\in S$. Then if we scale the tiling of the triangle by a factor of $q$, every vertex of each triangular tile has its $\Psi$-coordinates in the form $qx$ for some $x$ in $S$, and hence are closer than $\frac 14$ to an integer. Also, $T$ has bottom side second $\Psi$-coordinate equal to $0$ and top vertex second $\Psi$-coordinate $q$, both integers. Therefore by Corollary \ref{triangcorollary}, all triangular tiles have their side lengths equal to an integer.
\end{proof}

\comment{
As in Section~\ref{s:rescaling} we can also obtain similar corollaries in an analogous way. Nevertheless, this following last result is unique in a triangular tiling, as a consequence of the previous theorem, in the sense that there is 
no
analogous version for it in the rectangular tiling cases.
}

Analogous results for tilings of isosceles trapezoids and parallelograms follow easily.

\begin{corollary}
For any isosceles trapezoid or parallelogram tiled by $n$ equilateral triangles,
one can find a scaling of it with longest side length at most $4^{\frac{2n}{3}}$ for the trapezoid and $4^{\frac{2n+2}{3}}$ for the parallelogram, such that all triangular tiles have integer side 
lengths.
\end{corollary}

\begin{proof}
By adding one more equilateral triangle tile, with side length equal to the smaller base of the trapezoid and joining it to this base, we create an equilateral triangle tiling with $n+1$ triangular tiles. So, by making use of Theorem \ref{t:strongt}, we get the $4^{\frac{2n}{3}}$ bound. The parallelogram case is almost the same, but we need two more tiles instead of one in order to create an equilateral triangle tiling with $n+2$ tiles, thus obtaining the $4^{\frac{2n+2}{3}}$ bound.
\end{proof}

\centering{A\footnotesize CKNOWLEDGEMENTS} \\
\hfill

{The authors are grateful to the anonymous referees for their very careful reading of the manuscript and for their constructive suggestions that improved the presentation of the paper.}



\end{document}